\documentclass[a4paper,12pt]{article}
\usepackage{amsmath,amsthm,amssymb}
\usepackage{tabularx}

\newcommand{\C}{{\mathbb C}}

\newcommand{\R}{{\mathbb R}}

\newcommand{\wt}{{\rm wt}}

\newcommand{\GL}{{\rm GL}}

\DeclareMathOperator{\Harm}{Harm}

\newtheorem{Thm}{Theorem}[section]
\newtheorem{Lem}[Thm]{Lemma}

\newtheorem{Prop}[Thm]{Proposition}

\theoremstyle{definition}
\newtheorem{Def}[Thm]{Definition}
\newtheorem{Rem}[Thm]{Remark}

\makeatletter
\@addtoreset{equation}{section}

\makeatother

\begin{document}

\title{An upper bound of the value of $t$ of the support $t$-designs of extremal binary doubly even 
self-dual codes}

\author{
Tsuyoshi Miezaki\thanks{
Faculty of Education, Art and Science, 
Yamagata University, 
Yamagata 990--8560, Japan. 
email: miezaki@e.yamagata-u.ac.jp
}
and 
Hiroyuki Nakasora\thanks{
Graduate School of Natural Science and Technology, 
Okayama University, Okayama 700-8530, Japan. 
email: nakasora@math.okayama-u.ac.jp
}
}

\date{}

\maketitle

\begin{abstract}
Let $C$ be an extremal binary doubly even self-dual code of length $n$ 
and $s(C)$ denote the largest integer $t$ such that the support design of $C$ holds a $t$-design for some weight. 
In this paper, we prove $s(C) \leq 7$.
\end{abstract}












\setcounter{section}{+0}
\section{Introduction}

Let $C$ be an extremal binary doubly even self-dual code (Type II code) of length $n$. 
It was shown by Zhang \cite{Zhang(1999)} that $C$ does not exist if $n=24m$ ($m \geq 154)$, 
$24m+8$ ($m \geq 159$), $24m+16$ ($m \geq 164$).
A $t$-$(v,k,{\lambda})$ design is a pair 
$\mathcal{D}=(X,\mathcal{B})$, where $X$ is a set of points of 
cardinality $v$, and $\mathcal{B}$ a collection of $k$-element subsets
of $X$ called blocks, with the property that any $t$ points are 
contained in precisely $\lambda$ blocks.
It follows that every
$i$-subset of points $(i \leq t)$ is contained in exactly
$\lambda_{i}= \lambda \binom{v-i}{t-i} / \binom{k-i}{t-i}$ blocks. 
The support $\mbox{supp} (c)$ of a codeword $c=(c_{1}, \dots, c_{n}) \in C$ 
is 
the set of indices of its nonzero coordinates: $\mbox{supp} (c) = \{ i : c_{i} \neq 0 \}$. 
The support design of $C$ for a given nonzero weight $w$ 
($w \equiv 0 \pmod 4$ and $4 \lfloor n/24 \rfloor +4 \leq w \leq n-(4 \lfloor n/24 \rfloor +4)$)
is the design 
for which the points are the $n$ coordinate indices, and the blocks are the supports of all codewords of weight $w$. 
Let $D_{w}$ be the support design of $C$ for a weight $w$. 
Then it is known from the Assmus--Mattson theorem \cite{assmus-mattson} 
that $D_{w}$ of all weights becomes a $5$-, $3$- and $1$-design for $n=24m$, $24m+8$ and $24m+16$, 
respectively. 
Note that no $t$-design for $t \geq 6$ has yet been obtained 
from the support designs for codes. 

%
Let 
\begin{align*}
s(C)&:=\max\{t\in \mathbb{N}\mid \exists w; \mbox{ s.t.~} 
D_{w} \mbox{ is a } t\mbox{-design}\},\\
\delta(C)&:=\max\{t\in \mathbb{N}\mid \forall w; 
D_{w} \mbox{ is a } t\mbox{-design}\}. 
\end{align*}
Note that $\delta(C)\leq s(C)$. 
The following theorem gives the lower bound of $\delta(C)$ due to Janusz \cite{Janusz}. 

\begin{Thm}\label{thm:Janusz}
Let $C$ be an extremal binary doubly even self-dual code of length $n=24m+8r$, $r=0,1$ or $2$.
Then either $\delta(C) \geq 7-2r$, or 
$\delta(C)= 5-2r$ and there is no nontrivial weight $w$ such that $D_{w}$ holds a $(1+ \delta(C))$-design. 
\end{Thm}

In this paper, we investigate an upper bound of $s(C)$. 
First, we collect some known results for the support $t$-design of the minimum weight. 
Let $D_{4m+4}^{24m}$ be the support $t$-designs of the minimum weight ($w=4m+4$) of 
an extremal binary doubly even self-dual $[24m,12m,4m+4]$ code. 
By the Assmus--Mattson theorem, $D_{4m+4}^{24m}$ is a $5$-$(24m,4m+4,\binom{5m-2}{m-1})$ design. 
Suppose that $D_{4m+4}^{24m}$ is a $t$-$(24m,4m+4,\lambda_{t})$ design with $t \geq 6$. 
Then $\lambda_{t}= \binom{5m-2}{m-1} \binom{4m-1}{t-5}/\binom{24m-5}{t-5}$ is a nonnegative integer. 
It is known that if $D_{4m+4}^{24m}$ is a $6$-design, then it is a $7$-design 
by a strengthening of the Assmus--Mattson theorem \cite{strengthening of the Assmus-Mattson theorem}.
In 2006, Bannai et al.~\cite{Bannai-Koike-Shinohara-Tagami} showed that $D_{4m+4}^{24m}$ is never a $9$-design.
In \cite{extremal design H-M-N}, we give the following theorem.

\begin{Thm}
[\cite{extremal design H-M-N,errataHMN}]
\label{thm:minimum weight bound H-K-N}
Let $D_{4m+4}^{24m}$, $D_{4m+4}^{24m+8}$ and $D_{4m+4}^{24m+16}$ be the support $t$-designs of the minimum weight of 
an extremal binary doubly even self-dual $[24m,12m,4m+4]$ code $($$m \leq 153$$)$,
$[24m+8,12m+4,4m+4]$ code $($$m \leq 158$$)$ and $[24m+16,12m+8,4m+4]$ code $($$m \leq 163$$)$, respectively.

\begin{enumerate}
\item[$(1)$] If $D_{4m+4}^{24m}$ becomes a $6$-design, then $D_{4m+4}^{24m}$ is a $7$-design and $m$ must be in the set 
$\{15$, $52$, $55$, $57$, 
$59$, $60$, $63$, $90$, $93$, $104$, 
$105$, $107$, $118$,
$125$, $127$, $135$, $143$, $151 \}$.
Finally, 
$D_{4m+4}^{24m}$ is never an $8$-design.

\item[$(2)$] 
If $D_{4m+4}^{24m+8}$ becomes a $4$-design, then $D_{4m+4}^{24m+8}$ is a $5$-design and $m$ must be in the set 
$\{15$, $35$, $45$, $58$, $75$, $85$, $90$, $95$, $113$, $115$, $120$, $125 \}$. 
If $D_{4m+4}^{24m+8}$ becomes a $6$-design, then $m$ must be $58$. 
If $D_{4m+4}^{24m+8}$ becomes a $7$-design, then $m$ must be $58$. 
Finally, $D_{4m+4}^{24m+8}$ is never an $8$-design.

\item[$(3)$] 
If $D_{4m+4}^{24m+16}$ becomes a $2$-design, then $D_{4m+4}^{24m+16}$ is a $3$-design and $m$ must be in the set 
$\{5$, $10$, $20$, $23$, $25$, $35$, $44$, $45$, $50$, $55$, 
$60$, $70$, $72$, $75$, $79$, 
$80$, $85$, $93$, $95$, $110$, $118$, $120$, 
$121$, $123$, $125$, $130$, $142$, $144$, 
$145$, $149$, $150$, $155$, $156$, $157$, 
$160$, $163$$\}$. 
If $D_{4m+4}^{24m+16}$ becomes a $4$-design, then $m$ must be in the set 
$\{10$, $23$, $79$, 
$93$, $118$, $120$, 
$123$, $125$, $142$$\}$. 
If $D_{4m+4}^{24m+16}$ becomes a $5$-design, then $m$ must be in the set 
$\{23$, $79$, 
$93$, $118$, $120$, 
$123$, $125$, $142$$\}$.
Finally, $D_{4m+4}^{24m+16}$ is never a $6$-design.

\end{enumerate}
\end{Thm}

We investigate the support designs of 
the non minimum weights and 
as a corollary, we have an upper bound of $s(C)$. 

This paper is organized as follows. In Section~\ref{sec:Har}, 
we give the definition and some properties 
of the harmonic weight enumerators, 
which are used to study the support designs for the non minimum weights. 
In particular, we remark that the harmonic weight enumerators of 
Type II codes relate some invariant rings of the finite subgroup of 
$GL(2,\C)$. 
Using this facts, in order to show our results, we extend the methods of Bachoc \cite{Bachoc} and Bannai et al.~\cite{Bannai-Koike-Shinohara-Tagami}. 
Our results is the following theorem. 

\begin{Thm}\label{thm:main upper bound}
Let $C$ be an extremal binary doubly even self-dual code of length $n$.

\begin{enumerate}
\item[$(1)$] If $n=24m$, then $\delta(C)=s(C)=5$ or $\delta(C)=s(C)=7$.
\item[$(2)$] If $n=24m+8$, then $\delta(C)=s(C)=3$ or $5 \leq \delta(C) \leq s(C) \leq 7$.
\item[$(3)$] If $n=24m+16$, then $\delta(C)=s(C)=1$ or $3 \leq \delta(C) \leq s(C) \leq 5$.
\end{enumerate}

\end{Thm}


Thus we conclude that $\delta(C) \leq s(C) \leq 7$ for any extremal Type II code $C$.

\section{Harmonic weight enumerators}\label{sec:Har}

\subsection{Harmonic weight enumerators}\label{subsec:hw}

In this section, we extend a method of 
the harmonic weight enumerators 
which were used by Bachoc \cite{Bachoc} and 
Bannai et al.~\cite{Bannai-Koike-Shinohara-Tagami}.
For the readers convenience we quote from \cite{Bachoc,Delsarte}
the definitions and properties of discrete harmonic functions 
(for more information the reader is referred to \cite{Bachoc,Delsarte}). 


Let $\Omega=\{1, 2,\ldots,n\}$ be a finite set (which will be the set of coordinates of the code) and 
let $X$ be the set of its subsets, while, for all $k= 0,1, \ldots, n$, $X_{k}$ is the set of its $k$-subsets.
We denote by $\R X$, $\R X_k$ the free real vector spaces spanned by respectively the elements of $X$, $X_{k}$. 
An element of $\R X_k$ is denoted by
$$f=\sum_{z\in X_k}f(z)z$$
and is identified with the real-valued function on $X_{k}$ given by 
$z \mapsto f(z)$. 

Such an element $f\in \R X_k$ can be extended to an element $\tilde{f}\in \R X$ by setting, for all $u \in X$,
$$\tilde{f}(u)=\sum_{z\in X_k, z\subset u}f(z).$$
If an element $g \in \R X$ is equal to some $\tilde{f}$, for $f \in \R X_{k}$, we say that $g$ has degree $k$. 
The differentiation $\gamma$ is the operator defined by linearity from 
$$\gamma(z) =\sum_{y\in X_{k-1},y\subset z}y$$
for all $z\in X_k$ and for all $k=0,1, \ldots n$, and $\Harm_{k}$ is the kernel of $\gamma$:
$$\Harm_k =\ker(\gamma|_{\R X_k}).$$

\begin{Thm}[\cite{Delsarte}]\label{thm:design}
A set $\mathcal{B} \subset X_{k}$ of blocks is a $t$-design 
if and only if $\sum_{b\in \mathcal{B}}\tilde{f}(b)=0$ 
for all $f\in \Harm_k$, $1\leq k\leq t$. 
\end{Thm}

In \cite{Bachoc}, the harmonic weight enumerator associated to a binary linear code $C$ was defined as follows: 
\begin{Def}
Let $C$ be a binary code of length $n$ and let $f\in\Harm_{k}$. 
The harmonic weight enumerator associated to $C$ and $f$ is

$$W_{C,f}(x,y)=\sum_{c\in C}\tilde{f}(c)x^{n-\wt(c)}y^{\wt(c)}.$$
\end{Def}

Let $G$ be the subgroup of $\GL(2,\C)$ generated by elements 

$$T_{1}= \frac{1}{\sqrt{2}} \begin{pmatrix} 1 & 1 \\ 1 & -1 \end{pmatrix}, \ 
T_{2}= \begin{pmatrix} 1 & 0 \\ 0 & i \end{pmatrix}.$$
We consider the group $G= \langle T_{1}, T_{2} \rangle$ together with the characters $\chi_{k}$ defined by 
$$\chi_{k}(T_{1})=(-1)^{k},\  \chi_{k}(T_{2})=i^{-k}.$$
We denote by $I_{G}= \C [x,y]^{G}$ the ring of polynomial invariants of $G$ and by $I_{G,\chi_{k}}$ 
the ring of relative invariants of $G$ with respect to the character $\chi_{k}$. 
Let $P_{8}=x^8+14x^4y^4+y^8$, $P_{12}=x^2y^2(x^4-y^4)^2$, 
$P_{18}=xy(x^8-y^8)(x^8-34x^4y^4+y^8)$, 
$P_{24}=x^4y^4(x^4-y^4)^4$, $P_{30}=P_{12}P_{18}$ and 
\begin{align*}
I_{G,\chi_{k}}=
\left\{
\begin{array}{ll}
\langle P_8,P_{24}\rangle &\mbox{ if }k \equiv 0\pmod{4}\\
P_{12}\langle P_8,P_{24}\rangle &\mbox{ if }k \equiv 2\pmod{4}\\
P_{18}\langle P_8,P_{24}\rangle &\mbox{ if }k \equiv 3\pmod{4}\\
P_{30}\langle P_8,P_{24}\rangle &\mbox{ if }k \equiv 1\pmod{4}
\end{array}. 
\right. 
\end{align*}
Then the structure of these invariant rings is described as follows: 
\begin{Thm}[\cite{Bachoc}]\label{thm:invariant}
Let $C$ be an extremal binary doubly even self-dual code of length $n$, and let $f \in \Harm_{k}$. 
Then we have $W_{C,f}(x,y) =(xy)^{k} Z_{C,f} (x,y)$. Moreover, the polynomial $Z_{C,f} (x,y)$ is degree of $n-2k$ 
and  is in $I_{G, \chi_{k}}$, the space of the relative invariants of $G$ with respect to the character $\chi_{k}$. 
\end{Thm}

We recall the slightly more general definition of the notion of a $T$-design, for a subset $T$ of $\{ 1,2, \ldots, n \}$: 
a set $\mathcal{B}$ of blocks is called a $T$-design if and only if $\sum_{b\in \mathcal{B}}\tilde{f}(b)=0$ 
for all $f\in \Harm_k$ and for all $k \in T$. 
By Theorem \ref{thm:design}, a $t$-design is 
a $T= \{1, \ldots, t \}$-design.
Let $W_{C,f}=\sum_{i=0}^{n}c_f(i)x^{n-i}y^i$. 
Then we note that $D_w$ is a $T$-design if and only if $c_f(w)=0$ 
for all $f\in \mbox{Harm}_j$ with $j\in T$. 
The following theorem is called a strengthening of the Assmus--Mattson theorem.

\begin{Thm}[
\cite{{strengthening of the Assmus-Mattson theorem}}]\label{thm:strengthening of the Assmus-Mattson theorem}

Let $D_{w}$ be the support design of an extremal binary doubly even self-dual code of length $n$. 
\begin{itemize}
\item If $n \equiv 0 \pmod{24}$, $D_{w}$ is a $\{1, 2, 3, 4, 5, 7\}$-design.
\item If $n \equiv 8 \pmod{24}$, $D_{w}$ is a $\{1, 2, 3, 5\}$-design. 
\item If $n \equiv 16 \pmod{24}$,  $D_{w}$ is a $\{1, 3\}$-design.
\end{itemize}

\end{Thm}

We remark that Bachoc gave an alternative proof of a strengthening of the Assmus--Mattson theorem 
in \cite[Theorem~4.2]{Bachoc}. 
Bannai--Koike--Shinohara--Tagami \cite[Theorem 6 and Remark 5]{Bannai-Koike-Shinohara-Tagami} 
proved the following theorem which is essentially an alternative proof of Theorem~\ref{thm:Janusz}.

\begin{Thm}[\cite{Bannai-Koike-Shinohara-Tagami}]\label{Thm:Bannai et al. thm6}
Let $D_{w}$ be the support design of an extremal binary doubly even self-dual code of length $n$. 
\begin{enumerate}
\item[$(1)$] If $n \equiv 0 \pmod{24}$, $D_{w}$ of all weights are $6$-designs or not simultaneously.
\item[$(2)$] If $n \equiv 8 \pmod{24}$, $D_{w}$ of all weights are $4$-designs or not simultaneously.
\item[$(3)$] If $n \equiv 16 \pmod{24}$, $D_{w}$ of all weights are $2$-designs or not simultaneously.
\end{enumerate}

\end{Thm}

\subsection{Harmonic weight enumerators of extremal Type II codes}

In this section, 
we give the explicit description of 
the harmonic weight enumerators of 
extremal Type II codes of for the particular cases, which will be needed 
in the proof of the our theorems in Section \ref{sec: proof of thm}. 
We set $n=24m+8r$ the length of a code $C$. 
\paragraph{\underline{Case $t=4$ and $r=2$.}}
Let us assume that $t=4$, and $C$ is an extremal binary doubly even self-dual code of 
length $n=24m+16$. 
Then by the Theorem~\ref{thm:invariant} we have $W_{C,f}(x,y) =c(f) (xy)^{4} Z_{C,f} (x,y)$, 
where $c(f)$ is a linear function from Harm$_{t}$ to $\R$ and $Z_{C,f} (x,y) \in I_{G,\chi_{0}}$.
By Theorem \ref{thm:invariant}, 
$Z_{C,f} (x,y)$ is written in the following form:
\[
Z_{C,f}(x,y) = \sum_{i=0}^{m}a_{i}P_{8}^{3(m-i)+1} P_{24}^{i}.
\]
Since the minimum weight of $C$ is $4m+4$, 
we have $a_{i}=0$ for $i \neq m$. 
Therefore, $W_{C,f}(x,y)$ is written in the following form: 
\begin{align}
W_{C,f}(x,y) &=c(f) (xy)^{4} P_{8} P_{24}^{m}  \notag \\
& =c(f) x^{4m+4}y^{4m+4} (x^{4}-y^{4})^{4m} 
(x^{8}+14x^{4}y^{4}+y^{8}). \label{eqn:W_ 4l}
\end{align}

The other cases are as follows. 

\paragraph{\underline{Case $t=5$ and $r=2$.}}

$W_{C,f}(x,y)$ is written in the following form: 
\begin{align*}
W_{C,f}(x,y) &=c(f) (xy)^{5} Z_{C,f} (x,y) \notag \\
             &=c(f) (xy)^{5} P_{30}  \sum_{i=0}^{m-1}a_{i}P_{8}^{3(m-i)-3} P_{24}^{i}. 
\end{align*}
If $C$ is extremal, then 
\begin{align}
W_{C,f}(x,y) &=c(f) (xy)^{5} P_{30}  P_{24}^{m-1} \notag \\
             &=c(f) x^{4m+4}y^{4m+4} (x^{4}-y^{4})^{4m-1}(x^{4}+y^{4}) \notag \\
& \hspace{15pt} (x^{8}-34x^{4}y^{4}+y^{8}). \label{eqn:W_ 4l+1}
\end{align}

\paragraph{\underline{Case $t=6$ and $r=1,2$.}} 

$W_{C,f}(x,y)$ is written in the following form: 
\begin{align*}
W_{C,f}(x,y) &=c(f) (xy)^{6} Z_{C,f} (x,y) \notag \\
             &=c(f) (xy)^{6} P_{12}  \sum_{i=0}^{m-1}a_{i}P_{8}^{3(m-i)-3+r} P_{24}^{i}. 
\end{align*}
If $C$ is extremal, then 
\begin{align}
W_{C,f}(x,y) &=c(f) (xy)^{6} P_{12}  P_{8}^{r} P_{24}^{m-1} \notag \\
             &=c(f) x^{4m+4}y^{4m+4} (x^{4}-y^{4})^{4m-2} (x^{8}+14x^{4}y^{4}+y^{8})^{r}. \label{eqn:W_ 4l+2}
\end{align}

\paragraph{\underline{Case $t=7$ and $r=1$.}} 

$W_{C,f}(x,y)$ is written in the following form: 
\begin{align*}
W_{C,f}(x,y) &=c(f) (xy)^{7} Z_{C,f} (x,y) \notag \\
             &=c(f) (xy)^{7} P_{18}  \sum_{i=0}^{m-1}a_{i}P_{8}^{3(m-i)-3} P_{24}^{i}. 
\end{align*}
If $C$ is extremal, then 
\begin{align}
W_{C,f}(x,y) &=c(f) (xy)^{7} P_{18} P_{24}^{m-1}\notag\\
&=c(f) x^{4m+4}y^{4m+4} (x^{4}-y^{4})^{4m-3}(x^{4}+y^{4})\notag\\
&\hspace{15pt}(x^{8}-34x^{4}y^{4}+y^{8}).\label{eqn:W_ 4l+3}
\end{align}

\paragraph{\underline{Case $t=8$ and $r=0,1$.}}

$W_{C,f}(x,y)$ is written in the following form: 
\begin{align*}
W_{C,f}(x,y) &=c(f) (xy)^{8} Z_{C,f} (x,y) \notag \\
&=c(f) (xy)^{8}  \sum_{i=0}^{m-1}a_{i}P_{8}^{3(m-i)-2+r} P_{24}^{i}. 
\end{align*}
If $C$ is extremal, then 
\begin{align}
W_{C,f}(x,y) &=c(f) (xy)^{8} P_{8}^{r+1} P_{24}^{m-1}  \notag \\
& =c(f) x^{4m+4}y^{4m+4} (x^{4}-y^{4})^{4m-4} 
(x^{8}+14x^{4}y^{4}+y^{8})^{r+1}. \label{eqn:W_ 4l,8}
\end{align}


\subsection{Coefficients of the harmonic weight enumerators of extremal Type II codes}

As we mentioned in Section \ref{subsec:hw}, 
it is important for the support designs of a code $C$ whether the coefficients of $W_{C,f}(x,y)$ are zero or not.
Therefore, we investigate it and show the the following lemmas.

\begin{Lem}\label{lem:poly. zero 1}
Let $Q=(x^{4}-y^{4})^{\alpha} (x^{8}+14x^{4}y^{4}+y^{8})^{\beta}$  
with $0\leq \alpha \leq 652$ and $\beta= 1 ,2$.

\begin{enumerate}
\item[$(1)$] In the case $\beta =1$, 
if the coefficients of 
$(x^{4})^{\alpha+2-i}(-y^{4})^{i}$ in $Q$ are equal to $0$ 
for $0 \leq i \leq \frac{\alpha+2}{2}$, 
then $(\alpha, i)=(14,1)$, $(223,15)$.

\item[$(2)$]  In the case $\beta =2$, 
the coefficients of 
$(x^{4})^{\alpha+4-i}(-y^{4})^{i}$ in $Q$ are equal to $0$ 
for $0 \leq i \leq \frac{\alpha+4}{2}$, 
then $(\alpha, i)=(28,1)$. 

\end{enumerate}

\end{Lem}
\begin{proof}
We recall that 
$C$ does not exist if $n=24m$ ($m \geq 154)$, 
$24m+8$ ($m \geq 159$), $24m+16$ ($m \geq 164$).
Then $0\leq \alpha \leq 652$ satisfy the condition 
for $m \leq 163$ in equations (\ref{eqn:W_ 4l})--(\ref{eqn:W_ 4l,8}). 

$(1)$ In the case $\beta =1$, 
\begin{align*}
Q = & (x^{4}-y^{4})^{\alpha}(x^{8}+14x^{4}y^{4}+y^{8}) \\
  = & \sum_{j=0}^{\alpha} \binom{\alpha}{j} (-1)^{j}(x^{4})^{\alpha-j}(y^{4})^{j}(x^{8}+14x^{4}y^{4}+y^{8}).
\end{align*}

For $0\leq \alpha \leq 2$, it is easily seen by a direct computation that 
the all coefficients of $(x^{4})^{\alpha+2-i}(-y^{4})^{i}$ in $Q$ are not equal to $0$. 
Therefore, let $\alpha \geq 3$.
It is clear that the coefficient of 
$(x^{4})^{\alpha+2}$ in $Q$ (similarly $(y^{4})^{\alpha+2}$) is equals to $1$.
Next, the coefficient of 
$(x^{4})^{\alpha+1}(y^{4})$ in $Q$ (similarly $(-1)^{\alpha}(x^{4})(y^{4})^{\alpha+1}$) is equals to $14- \binom{\alpha}{1}$. 
Hence if $\alpha=14$, this coefficient is equals to $0$.  

The coefficients of 
$(x^{4})^{\alpha-(j-1)}(y^{4})^{j+1}$ in $Q$ for $1 \leq j \leq \alpha-1$ are the following formula: 

\begin{align*}
&\binom{\alpha}{j-1} (-1)^{j-1}(x^{4})^{\alpha-(j-1)}(y^{4})^{j-1} \times y^{8} \\
&+\binom{\alpha}{j} (-1)^{j}(x^{4})^{\alpha-j}(y^{4})^{j} \times 14x^{4}y^{4}  \\
&+\binom{\alpha}{j+1} (-1)^{j+1}(x^{4})^{\alpha-(j+1)}(y^{4})^{j+1} \times x^{8} \\
&= \left( \frac{j}{\alpha-j+1} -14 + \frac{\alpha-j}{j+1} \right)
\binom{\alpha}{j} (-1)^{j-1}(x^{4})^{\alpha-(j-1)}(y^{4})^{j+1}.
\end{align*}
If 
\begin{align*}
\left( \frac{j}{\alpha-j+1} -14 + \frac{\alpha-j}{j+1} \right)=\frac{16j^{2}-16 \alpha j + \alpha^{2}-13 \alpha -14}{(\alpha-j+1)(j+1)}=0, 
\end{align*}
we have $16j^{2}-16 \alpha j + \alpha^{2}-13 \alpha -14=0$ and 
\begin{align*}
j= \frac{2 \alpha \pm \sqrt{(3\alpha+7)(\alpha+2)}}{4}.
\end{align*}





Since $j$ is a nonnegative integer, $(3\alpha+7)(\alpha+2)$ is a square number. 
Let $\ell$ be the greatest common divisor of $3\alpha+7$ and $\alpha+2$. 
We set $3\alpha+7= \ell z_{1}$ and $\alpha+2 =\ell z_{2}$, 
where $z_{1}$, $z_{2}$ are nonnegative integers. 
Then we have $\ell(3z_{2}-z_{1})=-1$. Hence we have $\ell=1$. 
Therefore, both $3\alpha+7$ and $\alpha+2$ are square numbers. 

Let 
\begin{align}
\begin{cases}
3\alpha+7 = X^{2}  \\
\alpha+2  =Y^{2}  \label{eqn: alpha+2}
\end{cases},
\end{align}
where $X$ and $Y$ are nonnegative integers. 
Then we have 
\begin{equation}
X^{2}-3Y^{2}=1.  \label{eqn:pell eq}
\end{equation}
This equation is the instance of Pell equation $X^{2}-nY^{2}=1$ for $n=3$. 
The solving of the Pell equation is well known \cite[page 276]{IR}.
The equation (\ref{eqn:pell eq}) has the nontrivial smallest integer solution $(X_{1},Y_{1})=(2,1)$. 
Then all remaining solutions may be calculated as 
$X_{k}+Y_{k} \sqrt{3}=(X_{1}+Y_{1} \sqrt{3})^{k}=(2+\sqrt{3})^{k}$. 
Equivalently, we may calculate subsequent solutions via the recurrence relation
\begin{align*}
\begin{cases}
X_{k+1} =X_{1}X_{k}+3Y_{1}Y_{k} =2X_{k}+3Y_{k} \\
Y_{k+1} =Y_{1}X_{k}+ X_{1}Y_{k} = X_{k}+2Y_{k} 
\end{cases}.
\end{align*}


The above recurrence formulas generates the infinite sequence of solutions 
$$(2,1), (7,4), (26,15), (97,56), (362,209), \ldots.$$ 
By (\ref{eqn: alpha+2}), we obtain
$$\alpha= -1,14,223,3134,43679, \ldots.$$
Therfore the equation $\left( \frac{j}{\alpha-j+1} -14 + \frac{\alpha-j}{j+1} \right)=0$ 
has a solution $(\alpha, j)=(223,14)$ for $3 \leq \alpha \leq 652$ and $1\leq j\leq \alpha -1$. 

This completes the proof of Lemma \ref{lem:poly. zero 1} (1).

$(2)$ In the case $\beta =2$, 
\begin{align*}
Q = & (x^{4}-y^{4})^{\alpha}(x^{8}+14x^{4}y^{4}+y^{8})^{2} \\
  = & \sum_{j=0}^{\alpha} \binom{\alpha}{j} (-1)^{j}(x^{4})^{\alpha-j}(y^{4})^{j}
(x^{16}+28x^{12}y^{4}+198x^{8}y^{8}+28x^{4}y^{12}+y^{16}).
\end{align*}

For $0\leq \alpha \leq 4$, it is easily seen by a direct computation that 
the all coefficients of $(x^{4})^{\alpha+4-i}(-y^{4})^{i}$ in $Q$ are not equal to $0$. 
Therefore, let $\alpha \geq 5$.
It is clear that the coefficient of 
$(x^{4})^{\alpha+4}$ in $Q$ (similarly $(y^{4})^{\alpha+4}$) is equals to $1$.
The coefficient of 
$(x^{4})^{\alpha+3}(y^{4})$ in $Q$ (similarly $(-1)^{\alpha}(x^{4})(y^{4})^{\alpha+3}$) is equals to $28- \binom{\alpha}{1}$. 
Hence if $\alpha=28$, this coefficient is equals to $0$.  
The coefficient of 
$(x^{4})^{\alpha+2}(y^{4})^{2}$ in $Q$ (similarly $(-1)^{\alpha}(x^{4})^{2}(y^{4})^{\alpha+2}$) is equals to 
$198- 28 \binom{\alpha}{1}+ \binom{\alpha}{2}$. 
The equation $198- 28 \binom{\alpha}{1}+ \binom{\alpha}{2}=0$ has no integer solution. 
The coefficient of 
$(x^{4})^{\alpha+1}(y^{4})^{3}$ in $Q$ (similarly $(-1)^{\alpha}(x^{4})^{3}(y^{4})^{\alpha+1}$) is equals to 
$28- 198 \binom{\alpha}{1}+ 28 \binom{\alpha}{2}- \binom{\alpha}{3}$. 
The equation $28- 198 \binom{\alpha}{1}+ 28 \binom{\alpha}{2}- \binom{\alpha}{3}=0$ has no integer solution. 

The coefficients of 
$(x^{4})^{\alpha-(j-2)}(y^{4})^{j+2}$ in $Q$ for $2 \leq j \leq \alpha-2$ are equal to  
$\left(\frac{j(j-1)}{(\alpha -j+2)(\alpha-j+1)} - \frac{28j}{\alpha-j+1} +198 
- \frac{28(\alpha-j)}{j+1} + \frac{(\alpha-j)(\alpha -j-1)}{(j+1)(j+2)} \right)
\binom{\alpha}{j} (-1)^{j-2}$.
By a computer search,  the equation \\
$\left(\frac{j(j-1)}{(\alpha -j+2)(\alpha-j+1)} - \frac{28j}{\alpha-j+1} +198 
- \frac{28(\alpha-j)}{j+1} + \frac{(\alpha-j)(\alpha -j-1)}{(j+1)(j+2)} \right)=0$ 
has no integer solution
for $5 \leq \alpha \leq 652$ and  $2 \leq j \leq \alpha -2$.

This completes the proof of Lemma \ref{lem:poly. zero 1} (2). 

\end{proof}

\begin{Lem}\label{lem:poly. zero 2}

Let $\alpha \geq 1$ and $R=(x^{4}-y^{4})^{\alpha}(x^{4}+y^{4})(x^{8}-34x^{4}y^{4}+y^{8})$. 
If the coefficients of 
$(x^{4})^{\alpha+3-i}(-y^{4})^{i}$ in $R$ are equal to $0$ for $0 \leq i \leq \frac{\alpha+3}{2}$, 
then $\alpha=2i-3$.
\end{Lem}
\begin{proof}

If $\alpha=1$, $R=x^{16}-34x^{12}y^{4}+34x^{4}y^{12}+y^{16}$.
In this case, if the coefficients of 
$(x^{4})^{\alpha+3-i}(-y^{4})^{i}$ in $R$ are equal to $0$, then $i=2$.
For $\alpha = 2$, it is easily seen by a direct computation that 
the all coefficients of $(x^{4})^{\alpha+2-i}(-y^{4})^{i}$ in $R$ are not equal to $0$. 
Therefore, let $\alpha \geq 3$.
It is clear that the coefficient of 
$(x^{4})^{\alpha+3}$ in $R$ (similarly $(y^{4})^{\alpha+3}$) is equals to $1$.
The coefficient of 
$(x^{4})^{\alpha+2}(y^{4})$ in $R$ (similarly $(-1)^{\alpha}(x^{4})(y^{4})^{\alpha+2}$) is equals to 
$-34- \binom{\alpha}{1} +1<0$. 
The coefficient of 
$(x^{4})^{\alpha+1}(y^{4})^{2}$ in $R$ (similarly $(x^{4})^{2}(y^{4})^{\alpha+1}$) is equals to 
$1-34-\binom{\alpha}{1}(1-34) +\binom{\alpha}{2}=33 (\alpha-1) +\binom{\alpha}{2} >0$.

The coefficients of 
$(x^{4})^{\alpha-(j-1)}(y^{4})^{j+1}$ in $(x^{4}-y^{4})^{\alpha}(x^{8}-34x^{4}y^{4}+y^{8})$ 
for $1 \leq j \leq \alpha-1$ are the following formula: 

\begin{equation*}
\left( \frac{j}{\alpha-j+1} +34 + \frac{\alpha-j}{j+1} \right)
\binom{\alpha}{j} (-1)^{j-1}(x^{4})^{\alpha-(j-1)}(y^{4})^{j+1}.
\end{equation*}

Hence the coefficients of 
$(x^{4})^{\alpha-(j-1)}(y^{4})^{j+2}$ in $R$ 
for $1 \leq j \leq \alpha-1$ are the following formula: \\
$\left(  \frac{j}{\alpha-j+1}  +34 + \frac{\alpha-j}{j+1} \right) 
\binom{\alpha}{j} (-1)^{j-1}(x^{4})^{\alpha-(j-1)}(y^{4})^{j+1} \times y^{4}$ \\
$+\left( \frac{j+1}{\alpha-(j+1)+1} +34 + \frac{\alpha-(j+1)}{(j+1)+1} \right) 
\binom{\alpha}{j+1} (-1)^{j} (x^{4})^{\alpha-j}(y^{4})^{j+2} \times x^{4}$ \\
$=\left(\frac{j}{{\alpha}-j+1}+33-33\frac{({\alpha}-j)}{j+1}-\frac{({\alpha}-j)({\alpha}-j-1)}{(j+1)(j+2)}\right)\binom{\alpha}{j} (-1)^{j-1}(x^{4})^{\alpha-(j-1)}(y^{4})^{j+2}$. 


Let $$J=\frac{j}{{\alpha}-j+1}+33-33\frac{({\alpha}-j)}{j+1}-\frac{({\alpha}-j)({\alpha}-j-1)}{(j+1)(j+2)}.$$   
Then we have
\begin{eqnarray*}
J &=& \frac{j}{{\alpha}-j+1}+33-33\frac{({\alpha}-j)}{j+1}-\frac{({\alpha}-j)({\alpha}-j-1)}{(j+1)(j+2)} \\
  &=& \frac{33\alpha -32j+33}{\alpha-j+1} -\frac{(\alpha-j)(\alpha +32j +65)}{(j+1)(j+2)} \\
  &=& \frac{33(\alpha -2j-1)+34j+66}{\alpha-j+1} -\frac{(\alpha-j)(\alpha -2j -1)+(\alpha-j)(34j +66)}{(j+1)(j+2)} \\
  &=& \frac{(\alpha-2j-1) \Big( 33(j+1)(j+2)-(\alpha-j+1)(\alpha-j) -(\alpha+2)(34j+66) \Big)}{(\alpha-j+1)(j+1)(j+2)} \\
  &=& \frac{(\alpha-2j-1) \Big( 32j^{2}-32(\alpha-1)j -(\alpha+1)(\alpha+66) \Big)}{(\alpha-j+1)(j+1)(j+2)}.
\end{eqnarray*}
Since $3 \leq \alpha \leq 652$ and $1 \leq j \leq \alpha-1$, we have $32j^{2}-32(\alpha-1)j=32 (j - \frac{\alpha-1}{2})^{2} -8(\alpha-1)^{2} \leq 0$. 
Then we have $32j^{2}-32(\alpha-1)j -(\alpha+1)(\alpha+66)<0$. 
Hence if $J=0$, $\alpha= 2j+1$.



\end{proof}

\section{Proof of Theorems}\label{sec: proof of thm}
\subsection{Case for $n=24m$}

In this section, we consider the case of length $n=24m$. 
Let $D_{w}^{24m}$ be the support $t$-design of weight $w$ of 
an extremal binary doubly even self-dual $[24m,12m,4m+4]$ code $($$m \leq 153$$)$.
By Theorem~\ref{thm:Janusz} and Theorem~\ref{thm:minimum weight bound H-K-N} (1), 
we remark that if there exists $w'$ such that 
$D_{w'}^{24m}$ becomes a $6$-design,  
then $D_{w}^{24m}$  is a $7$-design for any $w$, and 
$m$ must be in the set 
$\{15$, $52$, $55$, $57$, $59$, $60$, $63$, $90$, $93$, $104$, $105$, $107$, $118$, $125$, $127$, $135$, $143$, $151 \}$.

For $t \geq 8$, we give the following proposition. 

\begin{Prop}\label{prop:length 24m}
For any extremal binary doubly even self-dual code of length $n=24m$, 
the support designs of all weights are $8$-designs or not simultaneously. 
\end{Prop}
\begin{proof}
If $r=0$ in the equation (\ref{eqn:W_ 4l,8}), we have  
$$W_{C,f}(x,y)  =c(f) x^{4m+4}y^{4m+4} (x^{4}-y^{4})^{4m-4} (x^{8}+14x^{4}y^{4}+y^{8}).$$
We recall that 
$C$ does not exist if $n=24m$ ($m \geq 154)$ \cite{Zhang(1999)}. 
By Lemma~\ref{lem:poly. zero 1} (1), 
the coefficients of $x^{i}$ with $i \equiv 0 \pmod{4}$ and $4m+4 \leq i \leq n-(4m+4)$ are 
all nonzero if $c(f)\neq 0$ or zero if $c(f)=0$ for $m\leq 153$. 
Therefore, the support designs of all weights are $8$-designs or not simultaneously. 
\end{proof}

We apply the results of Theorem~\ref{thm:minimum weight bound H-K-N} (1)
to Proposition~\ref{prop:length 24m}.
Then we obtain the following theorem.

\begin{Thm}\label{thm:main thm 1}
 $D_{w}^{24m}$ is never an $8$-design for any $w$.
\end{Thm}

Thus the proof of Theorem \ref{thm:main upper bound} (1) is completed.

\subsection{Case for $24m+8$}

In this section, we state the cases of length $n=24m+8$. 
Let $D_{w}^{24m+8}$ be the support $t$-design of weight $w$ of 
an extremal binary doubly even self-dual $[24m+8,12m+4,4m+4]$ code ($m \leq 158$).
By Theorem~\ref{thm:Janusz} and Theorem~\ref{thm:minimum weight bound H-K-N} (2), 
we remark that if there exists $w'$ such that 
$D_{w'}^{24m+8}$ becomes a $4$-design,  
then $D_{w}^{24m+8}$ is a $5$-design for any $w$, and 
$m$ must be in the set 
$\{15$, $35$, $45$, $58$, $75$, $85$, $90$, $95$, $113$, $115$, $120$, $125 \}$.

For $t \geq 6$, we give the following proposition.  
We call $w$ the middle weight if $w= n / 2$.

\begin{Prop}\label{prop:length 24m+8}
Let $D_{w}^{24m+8}$ be the support $t$-design of weight $w$ of 
an extremal binary doubly even self-dual code of length $n=24m+8$.

\begin{enumerate}
\item[$(1)$] 

\begin{enumerate}
\item[{\rm (i)}] Assume that $m \neq 4$. 
Then $D_{w}^{24m+8}$ of all weights $w$ are $6$-designs or not simultaneously. 
\item[{\rm (i\hspace{-.1em}i)}] 
Assume that $m= 4$. \\
Then $D_{w}^{104}$ is a 
$
\begin{cases}
\{1,2,3,5 \}\text{-design} & \text{if}\  w \neq 24 \\
\{1,2,3,5,6 \}\text{-design} & \text{if}\  w= 24. 
\end{cases}
$
\end{enumerate}

\item[$(2)$] 
\begin{enumerate}
\item[{\rm (i)}] 
$D_{w}^{24m+8}$ of all weights $w$ except for $w = 12m+4$ are $7$-designs 
or not simultaneously. 
\item[{\rm (i\hspace{-.1em}i)}] 
$D_{12m+4}^{24m+8}$ is a $\{1,2,3,5,7 \}$-design.

\end{enumerate}

\item[$(3)$] 

\begin{enumerate}
\item[{\rm (i)}] Assume that $m \neq 8$. 
Then $D_{w}^{24m+8}$ of all weights $w$ are $8$-designs or not simultaneously. 
\item[{\rm (i\hspace{-.1em}i)}] 
Assume that $m= 8$. \\
Then $D_{w}^{200}$ is a 
$
\begin{cases}
\{1,2,3,5 \}\text{-design} & \text{if}\  w \neq 40 \\
\{1,2,3,5,8 \}\text{-design} & \text{if}\  w= 40. 
\end{cases}
$
\end{enumerate}


\end{enumerate}

\end{Prop}

\begin{proof}
$(1)$
If $r=1$ in the equation (\ref{eqn:W_ 4l+2}), we have  
$$W_{C,f}(x,y)  =c(f) x^{4m+4}y^{4m+4} (x^{4}-y^{4})^{4m-2} (x^{8}+14x^{4}y^{4}+y^{8}).$$
By Lemma~\ref{lem:poly. zero 1} (1), if $m \neq 4$, the coefficients of $x^{i}$ with $i \equiv 0 \pmod{4}$ and $4m+4 \leq i \leq n-(4m+4)$ are 
all nonzero or zero at the same time. 
Therefore, if $m \neq 4$, $D_{w}^{24m+8}$ of all weights $w$ are $6$-designs or not simultaneously. 

Let $m=4$. By Lemma~\ref{lem:poly. zero 1} (1), 
if $i \neq 24$, the coefficients of $x^{i}$ with $i \equiv 0 \pmod{4}$ and $20 \leq i \leq 84$ are 
all nonzero or zero at the same time. 
Also, the coefficient of $x^{24}$ is equals to $0$. 
Therefore, if $w \neq 24$, $D_{w}^{104}$ is a $\{1,2,3,5 \}$-design. 
Also, $D_{24}^{104}$ is a $\{1,2,3,5,6 \}$-design.

$(2)$
By the equation (\ref{eqn:W_ 4l+3}), we have  
$$W_{C,f}(x,y) =c(f) x^{4m+4}y^{4m+4} (x^{4}-y^{4})^{4m-3}(x^{4}+y^{4})(x^{8}-34x^{4}y^{4}+y^{8}).$$
By Lemma~\ref{lem:poly. zero 2}, if  $i \neq 12m+4$, the coefficients of $x^{i}$ with $i \equiv 0 \pmod{4}$ and $4m+4 \leq i \leq n-(4m+4)$ are 
all nonzero or zero at the same time. 
Therefore, the support designs of all weights except for the middle weight are $7$-designs or not simultaneously. 

We consider the case that $w$ is the middle weight. By Lemma~\ref{lem:poly. zero 2}, 
the coefficient of $x^{12m+4}$ is equals to $0$. 
Hence $D_{12m+4}^{24m+8}$ is a $\{1,2,3,5,7 \}$-design.

$(3)$
If $r=1$ in the equation (\ref{eqn:W_ 4l,8}), we have  
$$W_{C,f}(x,y)  =c(f) x^{4m+4}y^{4m+4} (x^{4}-y^{4})^{4m-4} (x^{8}+14x^{4}y^{4}+y^{8})^{2}.$$

By Lemma~\ref{lem:poly. zero 1} (2), if $m \neq 8$, 
the coefficients of $x^{i}$ with $i \equiv 0 \pmod{4}$ and $4m+4 \leq i \leq n-(4m+4)$ are 
all nonzero or zero at the same time. 
Therefore, if $m \neq 8$, the support designs of all weights are $8$-designs or not simultaneously. 

Let $m=8$. By Lemma~\ref{lem:poly. zero 1} (2), 
if $i \neq 40$, the coefficients of $x^{i}$ with $i \equiv 0 \pmod{4}$ and $36 \leq i \leq 164$ are 
all nonzero or zero at the same time.  
Also, the coefficient of $x^{40}$ is equals to $0$. 
Therefore, if $w \neq 40$, $D_{w}^{200}$ is a $\{1,2,3,5 \}$-design. 
Also, $D_{40}^{200}$ is a $\{1,2,3,5,8 \}$-design.
\end{proof}

\begin{Rem} 
In Lemma~\ref{lem:poly. zero 1} (1), the solution $(\alpha,i)=(223,15)$ 
corresponds to the polynomial $Q=(x^{4}-y^{4})^{223} (x^{8}+14x^{4}y^{4}+y^{8})$.
In the case $t=9$ and $r=1$, if $C$ is extremal, then the harmonic weight enumerator is
\begin{align}
W_{C,f}(x,y) &=c(f) (xy)^{9} P_{30}P_{8} P_{24}^{m-2}\notag\\
&=c(f) x^{4m+4}y^{4m+4} (x^{4}-y^{4})^{4m-5}(x^{4}+y^{4})\notag\\
&\hspace{15pt}(x^{8}+14x^{4}y^{4}+y^{8})(x^{8}-34x^{4}y^{4}+y^{8}). \label{eqn:W_ 9}
\end{align}

The polynomial $Q$ is contained in the case of $m=57$ in the equation (\ref{eqn:W_ 9}).  
By a computation, the coefficients of $x^{i}$ in the equation (\ref{eqn:W_ 9}) 
with $i \equiv 0 \pmod{4}$ and $4m+4 \leq i \leq n-(4m+4)$  
are not equal to $0$. 
Thus the solution $(\alpha,i)=(223,15)$ does not give a design.


\end{Rem}

We apply  the results of Theorem~\ref{thm:minimum weight bound H-K-N} (2)
to Proposition~\ref{prop:length 24m+8}.
Then we obtain the following theorem.

\begin{Thm}\label{thm:main thm 2}
Let $D_{w}^{24m+8}$ be the support $t$-design of weight $w$ of 
an extremal binary doubly even self-dual $[24m+8,12m+4,4m+4]$ code $($$m \leq 158$$)$.

\begin{enumerate}


\item[$(1)$] 
\begin{enumerate}

\item[{\rm (i)}]

In the case $w \neq 12m+4$. If $D_{w}^{24m+8}$ becomes a $6$-design, then $m$ must be $58$. 
If $D_{w}^{24m+8}$ becomes a $7$-design, then $m$ must be $58$.

\item[{\rm (i\hspace{-.1em}i)}]
In the case $w = 12m+4$.
If $D_{12m+4}^{24m+8}$ becomes a $6$-design, 
then $D_{12m+4}^{24m+8}$ becomes a $7$-design and $m$ must be $58$. 

\end{enumerate}

\item[$(2)$] $D_{w}^{24m+8}$ is never a $8$-design for any $w$.

\end{enumerate}
\end{Thm}

Thus the proof of Theorem \ref{thm:main upper bound} (2) is completed. 

\subsection{Case for $24m+16$}

In this section, we state the cases of length $n=24m+16$. 
Let $D_{w}^{24m+16}$ be the support $t$-design of weight $w$ of 
an extremal binary doubly even self-dual $[24m+16,12m+8,4m+4]$ code ($m \leq 163$).
By Theorem~\ref{thm:Janusz} and Theorem~\ref{thm:minimum weight bound H-K-N} (3), 
we remark that if there exists $w'$ such that 
$D_{w'}^{24m+16}$ becomes a $2$-design, 
then $D_{w}^{24m+16}$ is a $3$-design for any $w$, and 
$m$ must be in the set 
$\{5$, $10$, $20$, $23$, $25$, $35$, $44$, $45$, $50$, $55$, 
$60$, $70$, $72$, $75$, $79$, 
$80$, $85$, $93$, $95$, $110$, $118$, $120$, 
$121$, $123$, $125$, $130$, $142$, $144$, 
$145$, $149$, $150$, $155$, $156$, $157$, 
$160$, $163$$\}$. 

For $t \geq 4$, we give the following proposition. 

\begin{Prop}\label{prop:length 24m+16}
Let $D_{w}^{24m+16}$ be the support $t$-design of weight $w$ of 
an extremal binary doubly even self-dual code of length $n=24m+16$.

\begin{enumerate}
\item[$(1)$] $D_{w}^{24m+16}$ of all weights $w$ are $4$-designs or not simultaneously. 

\item[$(2)$] 
\begin{enumerate}
\item[{\rm (i)}] 
$D_{w}^{24m+16}$ of all weights $w$ except for the middle weight 
$($$w = 12m+8$$)$ are $5$-designs or not simultaneously. 

\item[{\rm (i\hspace{-.1em}i)}] 
$D_{12m+8}^{24m+16}$ is a $\{1,2,3,5 \}$-design.
\end{enumerate}

\item[$(3)$] $D_{w}^{24m+16}$ of all weights $w$ are $6$-designs or not simultaneously. 

\end{enumerate}

\end{Prop}

\begin{proof}
$(1)$
If $r=2$ in the equation (\ref{eqn:W_ 4l}), we have  
$$W_{C,f}(x,y)  =c(f) x^{4m+4}y^{4m+4} (x^{4}-y^{4})^{4m} (x^{8}+14x^{4}y^{4}+y^{8}).$$
By Lemma~\ref{lem:poly. zero 1} (1), the coefficients of $x^{i}$ with $i \equiv 0 \pmod{4}$ and $4m+4 \leq i \leq n-(4m+4)$ are 
all nonzero or zero at the same time. 
Therefore, the support designs of all weights are $4$-designs or not simultaneously.

$(2)$
By the equation (\ref{eqn:W_ 4l+1}), we have  
$$W_{C,f}(x,y) =c(f) x^{4m+4}y^{4m+4} (x^{4}-y^{4})^{4m-1}(x^{4}+y^{4})(x^{8}-34x^{4}y^{4}+y^{8}).$$
By Lemma~\ref{lem:poly. zero 2}, if  $i \neq 12m+8$, the coefficients of $x^{i}$ with $i \equiv 0 \pmod{4}$ and $4m+4 \leq i \leq n-(4m+4)$ are 
all nonzero or zero at the same time. 
Therefore, the support designs of all weights except for the middle weight are $5$-designs or not simultaneously. 

We consider the case that $w$ is the middle weight. By Lemma~\ref{lem:poly. zero 2}, the coefficient of $x^{12m+8}$ is equals to $0$. 
Hence $D_{12m+8}^{24m+16}$ is a $\{1,2,3,5 \}$-design.

$(3)$
If $r=2$ in the equation (\ref{eqn:W_ 4l+2}), we have  
$$W_{C,f}(x,y)  =c(f) x^{4m+4}y^{4m+4} (x^{4}-y^{4})^{4m-2} (x^{8}+14x^{4}y^{4}+y^{8})^{2}.$$
By Lemma~\ref{lem:poly. zero 1} (2), the coefficients of $x^{i}$ with $i \equiv 0 \pmod{4}$ and $4m+4 \leq i \leq n-(4m+4)$ are 
all nonzero or zero at the same time. 
Therefore, the support designs of all weights are $6$-designs or not simultaneously. 

\end{proof}

We apply the results of Theorem~\ref{thm:minimum weight bound H-K-N} (3) 
to Proposition~\ref{prop:length 24m+16}. 
Then we obtain the following theorem.

\begin{Thm}\label{thm:main thm 3}
Let $D_{w}^{24m+16}$ be the support $t$-design of weight $w$ of 
an extremal binary doubly even self-dual $[24m+16,12m+8,4m+4]$ code $($$m \leq 163$$)$.

\begin{enumerate}




\item[$(1)$] 

\begin{enumerate}
\item[{\rm (i)}]
In the case $w \neq 12m+8$.
If $D_{w}^{24m+16}$ becomes a $4$-design,  
then $m$ must be in the set 
$\{10$, $23$, $79$, 
$93$, $118$, $120$, 
$123$, $125$, $142$$\}$. 

If $D_{w}^{24m+16}$ becomes a $5$-design,  
then $m$ must be in the set 
$\{23$, $79$, 
$93$, $118$, $120$, 
$123$, $125$, $142$$\}$.


\item[{\rm (i\hspace{-.1em}i)}]
In the case $w = 12m+8$.
If $D_{12m+8}^{24m+16}$ becomes a $4$-design, 
then $D_{12m+8}^{24m+16}$ becomes a $5$-design and $m$ must be in the set 
$\{10$, $23$, $79$, 
$93$, $118$, $120$, 
$123$, $125$, $142$$\}$. 

\end{enumerate}

\item[$(2)$] $D_{w}^{24m+16}$ is never a $6$-design for any $w$.

\end{enumerate}
\end{Thm}

Thus the proof of Theorem \ref{thm:main upper bound} (3) is completed.

\begin{Rem}
Let $\mathcal{D}=(X, \mathcal{B})$ be a $t$-design. 
The complementary design of $\mathcal{D}$ is $\bar{\mathcal{D}}=(X, \bar{\mathcal{B}})$, 
where $\bar{\mathcal{B}}= \{ X \setminus B : B \in \mathcal{B} \}$.
If $\mathcal{D}= \bar{\mathcal{D}}$, $\mathcal{D}$ is called a self-complementary design.
Let $D_{n/2}$ be the support $t$-design of the middle weight of 
an extremal binary doubly even self-dual code of length $n$. 
It is easily seen that $D_{n/2}$ is self-complement.


Alltop \cite{alltop} proved that
if $\mathcal{D}$ is a $t$-design with an even integer $t$ and self-complementary, 
then $\mathcal{D}$ is also a $(t+1)$-design.
Hence $D_{n/2}$ is a $\{ 1,3,5, \ldots, 2s+1\}$-design.
Thus Alltop's theorem gives an alternative proof of 
Propositions~\ref{prop:length 24m+8} (2) \rm (i\hspace{-.1em}i) 
and \ref{prop:length 24m+16} (2) \rm (i\hspace{-.1em}i).
\end{Rem}




\end{document}